\documentclass[12pt,a4paper]{article}

\usepackage[textwidth=2.2cm]{todonotes}
\usepackage{amsmath, amsfonts, amssymb}
\usepackage{amsthm}
\usepackage{geometry}
\usepackage{graphicx}
\usepackage[english]{babel}
\usepackage{parskip}
\usepackage{authblk}
\usepackage{hyperref}

\usepackage{csquotes}
\usepackage[style=alphabetic,sorting=nyt,backend=bibtex,%
            firstinits,maxnames=7,url=false]{biblatex}
\bibliography{biblio}

\title{A distance on curves modulo rigid transformations}
\author{Jaap Eldering\thanks{Corresponding author; email address
    \url{j.eldering@imperial.ac.uk}.}, Joris Vankerschaver}
\affil{\small Department of Mathematics, Imperial College London, \\
  \small London SW7 2AZ, United Kingdom}

\begingroup
    \makeatletter
    \@for\theoremstyle:=definition,remark,plain\do{%
        \expandafter\g@addto@macro\csname th@\theoremstyle\endcsname{%
            \addtolength\thm@preskip\parskip
            }%
        }
\endgroup

\newcommand{\D}{\ensuremath{\mathrm{D}}}
\newcommand{\T}{\ensuremath{\mathrm{T}}}

\newcommand{\R}{\ensuremath{\mathbb{R}}}

\newcommand{\slot}{\,\cdot\,}
\renewcommand{\d}{{\;\textrm{d}}}
\newcommand{\var}{\delta\!}

\newtheorem{thm}{Theorem}[section]
\newtheorem{prop}[thm]{Proposition}

\newtheorem{lemma}[thm]{Lemma}

\theoremstyle{remark}

\newtheorem{remark}[thm]{Remark}

\usepackage{mathtools}
\DeclarePairedDelimiter\abs{\lvert}{\rvert}
\DeclarePairedDelimiter\norm{\lVert}{\rVert}
\DeclarePairedDelimiterX\inner[2]{\langle}{\rangle}{#1 , #2}

\DeclareMathOperator{\Ad}{Ad}

\DeclareMathOperator{\tr}{tr}

\begin{document}
\maketitle

\begin{abstract}
  We propose a geometric method for quantifying the difference between
  parametrized curves in Euclidean space by introducing a distance
  function on the space of parametrized curves up to rigid
  transformations (rotations and translations). Given two curves, the
  distance between them is defined as the infimum of an energy
  functional which, roughly speaking, measures the extent to which the
  jet field of the first curve needs to be rotated to match up with
  the jet field of the second curve. We show that this energy
  functional attains a global minimum on the appropriate function
  space, and we derive a set of first-order ODEs for the minimizer.

\end{abstract}

\section{Introduction}

In this paper, we establish a new, geometric method for quantifying
the difference between curves in Euclidean space, based on the amount
of deformation needed to optimally match the jets of two curves.

In a nutshell, our method can be described as follows. Given two
curves $c_1, c_2$ defined on the same interval $I$, and with values in
$\R^n$, we try to rotate the tangent vector field of $c_1$ as well as
possible into that of $c_2$. As a first attempt, we may therefore look
for a family $g(s) \in O(n)$ of orthogonal transformations, so that
\begin{equation} \label{constraint}
  g(s) c_1'(s) = c_2'(s) \quad \text{for all $s \in I$},
\end{equation}
and quantify the distance between $c_1$ and $c_2$ as the ``magnitude''
of $g(s)$ in some appropriate sense. The hard
constraint~\eqref{constraint} cannot generally be satisfied, however,
since $O(n)$ preserves lengths. Therefore, we relax it into a soft
constraint and look instead for a curve $g(s)$ which
satisfies~\eqref{constraint} approximately, while at the same time
trying to minimize its variation. One way of doing so is by
considering a functional of the form
\begin{equation} \label{intro_energy}
  E[g; c_1, c_2] = \int_I \left(
    \norm{g(s) c_1'(s) - c_2'(s)}^2 +
    \norm{g(s)^{-1} g'(s)}^2 \right) \d{s},
\end{equation}
where $s \mapsto g(s)$ are curves in $O(n)$.
The distance between $c_1$ and $c_2$ we then define as
the infimum of $E$ over all possible curves $g(s)$:
\[
  d(c_1, c_2) = \inf_{g} E[g; c_1, c_2].
\]
Theorem~\ref{thm:distance} shows that this notion indeed defines a
distance, on the quotient space $C^k(I, \R^n) / E(n)$ of $k$-times
differentiable curves modulo rotations and translations. That is,
if $c_2$ can be obtained from $c_1$ by a rigid isometry, then the distance
vanishes. Conversely, if the distance is nonzero, then the curves are
not related by a rigid isometry.  More generally, the distance
function $d(c_1, c_2)$ is invariant (in both arguments separately) under
the action of the Euclidean group $E(n)$ on the space of curves.

A similar energy functional was considered in~\cite{HoNoVa2013} and we
employ a similar variational approach to characterize the minimizer
$g(s)$. However, whereas~\cite{HoNoVa2013} considered the action of
the Euclidean group $E(n)$ directly on the curves itself, we use
instead the action of the orthogonal group $O(n)$ on the tangent
vector field and the higher derivatives. The result is a notion of
(dis)similarity between curves which is a true distance function,
i.e.\ which is symmetric, non-negative, and satisfies the triangle
inequality. Furthermore, we prove that the energy functional $E$
attains its minimum in the space of curves $H^1(I,O(n))$, and that a
minimizer $g$ satisfies the variational equations in the strong sense,
so in particular $g \in C^2(I,O(n))$.

\subsection*{Plan of the paper}

In section~\ref{sec:curve_registration} we first define a slight
generalization of~\eqref{intro_energy}, in which not just the tangent
vector field of $c_1, c_2$ but also higher-order derivatives
(i.e.~jets of $c_1, c_2$) are taken into account, and in
section~\ref{sec:variational-eqns} we derive a set of variational
equations which characterize critical points of the energy and provide
some examples in section~\ref{sec:examples}. In
section~\ref{sec:exist-minimizer} we prove existence of
minimizers of $E$, and we show that any such minimizer must
necessarily be a solution of the variational equations. We finish the
paper by giving a probabilistic interpretation in
section~\ref{sec:probability} for the curve matching energy
functional.

\section{Geometric Preliminaries}

\subsection{Spaces of jets and their duals}

We consider a given parametrized curve $c\colon I \to \R^n$
with $I$ a bounded, closed interval
and we denote by $j^{1, k} c$ the $(1, k)$-jet of this
curve.  By definition,
\[
	(j^{1, k}_s c)(\sigma) = \sum_{i = 1}^k \frac{1}{i!} \frac{d^i c}{d s^i}(s) (\sigma - s)^i,
\]
which is simply the $k$-th order Taylor expansion at the point $s \in I$
without the zeroth order term.  For each $s$ we may view $j^{1,k}_s c$
as an $n \times k$-matrix, whose $i$-th column is the $i$-th
derivative of $c$ evaluated at $s$:
\[
	j^{1, k}_s c =
		\begin{pmatrix}
			\vline &          &\vline \\
			c'(s) & \cdots & c^{(k)}(s) \\
			\vline &          &\vline
		\end{pmatrix}.
\]
We let $V$ be the vector space of $n \times k$-matrices, so that
$j^{1, k} c$ can be viewed as a parametrized curve on the interval
$I$ with values in $V$.  For future reference, we identify the
dual $V^\ast$ with $V$ itself by means of the (Frobenius) inner
product
\[
	\inner{A}{B}_V = \tr(A^T B),
\]
for $A, B \in V \cong V^\ast$.  Secondly, we endow $V$ itself with a
weighted inner product, defined as follows. For constants
$\lambda_i > 0$, $i = 1, \ldots, k$, we let
\begin{equation} \label{eq:jet_inner_product}
	\inner{A}{B}_L
  = \sum_{i=1}^k \sum_{j=1}^n \lambda_i A_{ji} B_{ji}
  = \sum_{i=1}^k              \lambda_i \inner{a^{(i)}}{b^{(i)}}
  = \tr(L A^T B),
\end{equation}
where $L$ is the diagonal matrix with entries $L_{ii} = \lambda_i$,
$i = 1, \ldots, k$, while $a^{(i)}$ and $b^{(i)}$, $i = 1, \ldots, k$,
are the columns of $A$ resp. $B$. We denote the norm induced by the
inner product~\eqref{eq:jet_inner_product} by $\norm{\slot}_L$, and
the induced flat operation by $\flat\colon V \to V^\ast$, which is
given by $A \mapsto AL$.

\subsection{The action of the orthogonal group on the space of jets}

The Euclidean group $E(n) := O(n) \ltimes \R^n$ acts point-wise by
rotations and translations on curves in $\R^n$ and induces an action
by its subgroup $O(n)$ on the $(1, k)$-jets of these curves, in the
following way.  Let $c\colon I \to \R^n$ be a curve and
consider an element $(g, x)$ of the Euclidean group $E(n)$.  By
point-wise multiplication, we obtain a transformed curve, $d$, defined
by $d(s) := g c(s) + x$, for all $s \in I$. Note that the
derivatives of $d$ are given by $d^{(i)}(s) = g c^{(i)}(s)$. In other
words, the translational part of the action drops out, and we are left
with an action of $O(n)$ on the jet space $V \cong J^{1, k}(\R^n)$,
given by
\[
	g \cdot A =
		\begin{pmatrix}
			\vline &          &\vline \\
			g a^{(1)} & \ldots & g a^{(k)} \\
			\vline &          &\vline
		\end{pmatrix}
	= g A,
\]
where the operation on the right-hand side is simply matrix
multiplication, and the $a^{(i)}$ are the columns of $A$.

The family of norms $\norm{\slot}_{L}$ which we introduced previously
is natural in the sense that (for each set of constants $\lambda_i$),
the norm $\norm{\slot}_{L}$ is
invariant under the action of $O(n)$.  This is clear from the
expression~\eqref{eq:jet_inner_product}, where the inner product
$\inner{\slot}{\slot}_L$ is expressed as a linear combination of the
Euclidean inner products on each of the $k$ derivatives of the curve,
each of which is individually $O(n)$-invariant.

We recall that the Lie algebra $\mathfrak{o}(n)$ of the orthogonal
group consists of all antisymmetric $n \times n$-matrices $\Omega$,
equipped with the Lie bracket
$[\Omega, \Omega'] = \Omega \Omega' - \Omega' \Omega$.
The infinitesimal action of $\mathfrak{o}(n)$ on $V$
is given again by left matrix multiplication: any $\Omega \in
\mathfrak{o}(n)$ defines a linear transformation on $V$ given by
mapping $A \in V$ to $\Omega A$.  The Lie algebra $\mathfrak{o}(n)$ is
equipped with a positive-definite inner product defined by
\begin{equation} \label{eq:killing_inner_product}
  \inner{\Omega}{\Omega'}_{\mathfrak{o}(n)}
  :=  \tr( \Omega^T \, \Omega' )
   = -\tr( \Omega \, \Omega' )
\end{equation}
and we use this inner product to identify the dual space
$\mathfrak{o}(n)^\ast$ with $\mathfrak{o}(n)$ itself via matrix
transposition.

\subsection{The momentum map}

The action of a Lie group $G$ on a manifold $M$ induces a (cotangent
lift) momentum map $J\colon \T^\ast M \to \mathfrak{g}^\ast$. In our
case, where $G = O(n)$ acts by linear transformations on the vector
space $V$, we find that $J$ is a bilinear map, and denote it by the
\emph{diamond operator},
$\diamond \colon V \times V^\ast \to \mathfrak{g}^\ast$. To ease later
notation, let $\flat\colon V \to V^\ast$ denote the identification of
dual spaces induced by $\inner{\slot}{\slot}_L$.
Then, for each $A,B \in V$, $A \diamond B^\flat$ is an element of
$\mathfrak{o}(n)^\ast$, defined by
\begin{equation} \label{eq:diamond_operator_definition}
	\inner{A \diamond B^\flat}{\Omega}_{\mathfrak{o}(n)^\ast \times \mathfrak{o}(n)}
		= \inner{B^\flat}{\Omega A}_{V^\ast \times V}
\end{equation}
for all $\Omega \in \mathfrak{o}(n)$. Using the
expressions~\eqref{eq:jet_inner_product}
and~\eqref{eq:killing_inner_product} for the inner products, we may
rewrite this definition as
\[
    \tr((A \diamond B^\flat) \Omega)
  = \tr((L B^T)(\Omega A))
  = \tr(A L B^T \Omega),
\]
and as this must hold for all $\Omega \in \mathfrak{o}(n)$, we have
that
\begin{equation} \label{eq:expression_diamond}
	A \diamond B^\flat = (ALB^T)_{\mathrm{antisymm}} =
            \frac{1}{2} ( ALB^T - BLA^T).
\end{equation}

From the previous expression, or by direct inspection, it is easy to
derive the following result.
\begin{lemma} \label{lem:antisymmetry}
  The diamond operator~\eqref{eq:diamond_operator_definition} composed
  with the flat operator is antisymmetric, viz.
  $A \diamond B^\flat = - B \diamond A^\flat$ for all $A, B \in V$.
\end{lemma}

\section{The curve registration functional}
\label{sec:curve_registration}

Given a source and target curve $c_1, c_2\colon I \to \R^n$,
we define an energy functional on the space of curves
$g\colon I \to O(n)$, given by
\begin{equation} \label{eq:energy_functional}
	E[g; c_1, c_2] = \frac{1}{2} \int_I \left(
		  \norm{ g(s) \cdot j^{1, k}_s c_1 - j^{1, k}_s c_2 }_L^2
		+ \norm{ g(s)^{-1} g'(s) }_{\mathfrak{o}(n)}^2
			\right) \d s.
\end{equation}
Here, the norms on the right-hand
side are induced by the inner products~\eqref{eq:jet_inner_product}
and~\eqref{eq:killing_inner_product}, respectively. The first term in
the energy functional measures how well the curve $s \mapsto g(s)$ is
able to rotate the jet field of $c_1$ into that of $c_2$, while the
second term is a measure for how far the curve $s \mapsto g(s)$ is
from being constant.

To simplify the notation somewhat for later, we let
\begin{equation}\label{eq:def-Q-Omega}
Q(s) := g(s) \cdot j^{1, k}_s c_1 - j^{1, k}_s c_2
\quad\text{and}\quad
\Omega(s) := g(s)^{-1} g'(s),
\end{equation}
so that
\[
	E[g; c_1, c_2] = \frac{1}{2} \int_I \left(
		  \norm{ Q(s) }_L^2
		+ \norm{ \Omega(s) }_{\mathfrak{o}(n)}^2
			\right) \d s.
\]

\begin{remark}
  Notice that a curve $g$ must have a square-integrable derivative
  for~\eqref{eq:energy_functional} to be well-defined. In
  section~\ref{sec:exist-minimizer} we shall properly define the space
  $H^1(I,O(n))$ of all such curves. For the moment, we can mostly
  ignore the technicalities associated with it: the space of smooth
  functions $C^\infty(I,O(n))$ is dense in $H^1(I,O(n))$, so
  minimizing over only smooth functions does not affect the distance
  defined below, and the variational equations derived in
  section~\ref{sec:variational-eqns} turn out to always be $C^2$.
\end{remark}

We now define the distance function $d(c_1,c_2)$ as the infimum of
this action over all curves $g$. Remark that it still has to be
checked that this defines a proper distance function on curves modulo
rigid transformations.

\begin{thm}\label{thm:distance}
  The function
  \begin{equation}\label{eq:distance}
    d(c_1,c_2) = \inf_{g \in H^1(I,O(n))} E[g;c_1,c_2]
  \end{equation}
  defines a distance function on the space $C^k(I,\R^n) / E(n)$ of
  curves modulo rigid transformations.
\end{thm}

\begin{remark}
  We expect that the space of $C^k$ curves modulo rigid
  transformations has a completion under the
  distance~\eqref{eq:distance} to $H^k(I,\R^n)/E(n)$, i.e.\ curves
  whose $k$-th derivative is square-integrable, but we do not prove this.
  Note that since $g \in H^1$, it follows in particular that
  $g \in L^\infty$ and therefore its action on $j_s^{1,k} c_1$
  in~\eqref{eq:energy_functional} would still be well-defined.
\end{remark}

\begin{proof}
We first check that $d$ is well-defined as a function on
$C^k(I,\R^n) / E(n)$. If $c_1,c_2$ are rigidly equivalent,
then we have $g c_1(s) + x = c_2(s)$ for all $s \in I$ and a fixed
$(g, x) \in E(n)$. For the derivatives, we have that
$g \cdot j^{1, k}_s c_1 = j^{1, k}_s c_2$, and viewing this $g$ as a
constant function into $O(n)$, it is then immediately verified that
$E[g;c_1,c_2] = 0$, thus $d(c_1,c_2) = 0$. More generally, let
$\tilde{c}_i \in [c_i]$, $i=1,2$, that is, we have
$(g_i,x_i) \in E(n)$ such that $\tilde{c}_i = (g_i,x_i) \cdot c_i$ and
thus $j^{1,k} \tilde{c}_i = g_i \cdot j^{1,k} c_i$. Let
$g \in H^1(I,O(n))$ be arbitrary. If we set
\begin{equation}\label{eq:transform_g}
  \tilde{g}(s) = g_2 \, g(s) \, g_1^{-1},
\end{equation}
and use that $O(n)$ acts by isometry on $V$ and that the inner product
on $\mathfrak{o}(n)$ is $\Ad$-invariant, then we see that
\begin{align}
  E[\tilde{g};\tilde{c}_1,\tilde{c}_2]
  &= \frac{1}{2} \int_I \Big(
       \norm[\big]{(g_2 g(s) g_1^{-1}) j^{1,k} \tilde{c}_1 - j^{1,k} \tilde{c}_2}_L^2
      +\norm[\big]{(g_1 g(s)^{-1} g_2^{-1})(g_2 g'(s) g_1^{-1})}^2 \Big) \d s \notag\\
  &= \frac{1}{2} \int_I \Big(
       \norm[\big]{g_2\big(g(s) j^{1,k} c_1 - j^{1,k} c_2\big)}_L^2
      +\norm[\big]{g(s)^{-1} g'(s)}^2 \Big) \d s \notag\\
  &=  E[g;c_1,c_2]. \label{eq:equivalence_energy}
\end{align}
Since~\eqref{eq:transform_g} defines an isomorphism of $H^1(I,O(n))$, it
follows that the infima on both sides
of~\eqref{eq:equivalence_energy} are equal, so $d$ descends to the
quotient $C^k(I,\R^n)/E(n)$.

Let us check the distance properties. First of all, it is clear that
$d(c_1,c_2) \ge 0$. Secondly, if $c_1 = c_2$, then clearly choosing
$g(s) = e$ yields $d(c_1,c_2) = 0$. By the discussion above, the same
holds if $c_1$ and $c_2$ are rigidly equivalent.

Conversely, to prove that $d(c_1,c_2) > 0$ for any two curves
that are not rigidly equivalent, we require the assumption that the
norm weight of the first order jet is nonzero,\footnote{%
  We can relax the assumption to allow weights $\lambda_i \ge 0$ for $i \ge 2$. Then
  $\inner{\slot}{\slot}_L$ is not positive definite anymore on $V$,
  but the induced $\flat\colon V \to V^\ast$ is still well-defined.%
} i.e.~$\lambda_1 > 0$.
Otherwise we could choose $c_1(s) = 0$ and
$c_2(s) = s v$ with $v \in \R^n$ nonzero, and find that for $g(s) = e$
we have $\norm{j^{1,k}_s c_2 - g(s) \cdot j^{1,k}_s c_1} = 0$. Hence
we would have $d(c_1,c_2) = 0$, while the curves are not related by a
rigid transformation.

By Theorem~\ref{thm:exist-minimizer} there exists a minimizer
$g^* \in H^1(I,O(n))$ of the distance $d(c_1,c_2) = 0$ (note that
there is no circular dependency as this theorem does not depend on $d$
being a distance). This implies that $\norm{\Omega}^2 = 0$, hence
$\Omega = 0$ and $g(s) = g \in O(n)$ is constant. The fact that the
first term in~\eqref{eq:energy_functional} must also be zero then
implies that $c_2'(s) = g \cdot c_1'(s)$, hence $c_1,c_2$ are related
by a rigid transformation. Thus, $d(c_1,c_2) = 0$ if and only if
$c_1,c_2$ are rigidly equivalent.

Furthermore, we check that $d$ satisfies the triangle inequality.
Using the infimum definition, let $g,h \in H^1(I,O(n))$ be
approximate minimizers, i.e.~$E[g;c_1,c_2] \le d(c_1,c_2) + \epsilon$
and $E[h;c_2,c_3] \le d(c_2,c_3) + \epsilon$. We suppress the argument
$s$ to obtain
\begin{align*}
  d(c_1,c_3)
  &\le E[h \cdot g;c_1,c_3]\\
  &=   \frac{1}{2} \int_I \Big[ \norm[\big]{h g j^{1,k} c_1 - j^{1,k} c_3}_L^2
                               +\norm{g^{-1}h^{-1}(hg)'}^2 \Big] \d s\\
  &\le \frac{1}{2} \int_I \Big[
                    \Big(\norm[\big]{h(g j^{1,k} c_1 - j^{1,k} c_2)}_L^2
                        +\norm[\big]{  h j^{1,k} c_2 - j^{1,k} c_3}_L^2 \Big)\\
  &\hspace{1.5cm} + \Big(\norm{g^{-1}h^{-1}h'g}^2 + \norm{g^{-1}h^{-1}hg'}^2 \Big)\Big] \d s\\
  &= E[g;c_1,c_2] + E[h;c_2,c_3] \le d(c_1,c_2) + d(c_2,c_3) + 2\epsilon.
\end{align*}
Since such $g,h$ can be found for any $\epsilon > 0$, the triangle
inequality follows.

Symmetry of $d$ follows, since we have for any $g \in H^1(I,O(n))$ that
\begin{align*}
  E[g;c_1,c_2]
  &= \frac{1}{2} \int_I \Big( \norm[\big]{g j^{1,k} c_1 - j^{1,k} c_2}_L^2
                             +\norm{g^{-1}g'}^2 \Big) \d s\\
  &= \frac{1}{2} \int_I \Big( \norm[\big]{j^{1,k} c_1 - g^{-1} j^{1,k} c_2}_L^2
                             +\norm{-g'g^{-1}}^2 \Big) \d s
  =  E[g^{-1};c_2,c_1].
\end{align*}
This concludes the proof that $d$ is a distance on $C^k(I,\R^n)/E(n)$.
\end{proof}

\section{Variational equations for curve registration}
\label{sec:variational-eqns}

We now look for necessary conditions for a curve $g\colon I \to O(n)$
to be a critical point of the energy
functional~\eqref{eq:energy_functional}. Note that with the
notation~\eqref{eq:def-Q-Omega} the energy functional becomes
\[
E[g(\cdot)] = \frac{1}{2} \int_I \left(
    \norm{Q(s)}_L^2 + \norm{\Omega(s)}_{\mathfrak{o}(n)}^2
  \right) \d s.
\]

First of all, we have to consider curves $g \in H^2(I,O(n))$ since the
term $\Omega(s)$ will get differentiated when calculating the
variational equations, cf.~\cite[p.~247]{AbMa1978} for remarks and
references in case of Lagrangian mechanics. In
section~\ref{sec:exist-minimizer} we will address this issue further
and prove that a minimizer of $E$ is necessarily a function
$g \in C^2(I,O(n)) \subset H^1(I,O(n))$.

Consider a family of curves $g_\epsilon \in H^2(I,O(n))$,
which depends smoothly on the parameter $\epsilon$ in a small
neighborhood around $0$, and consider the effect on $E$ of varying
$\epsilon$ around $0$. For the first variation, we have
\[
	\delta E :=
		\frac{d}{d \epsilon} E(g_\epsilon) \Big|_{\epsilon = 0}
	= \int_I \Big( \inner{Q(s)}{\delta Q(s)}_L
	              +\inner{\Omega(s)}{\delta \Omega(s)}_{\mathfrak{o}(n)} \Big) \d s,
\]
and it now remains to express the variations $\delta Q$ and $\delta
\Omega$ in terms of $\displaystyle \delta g = \frac{d g_\epsilon}{d
  \epsilon}\Big|_{\epsilon=0}$. For the former, we have
\[
	\delta Q(s) = \delta g(s) j^{1, k}_s c_1
		= g(s) g(s)^{-1} \delta g(s) j^{1, k}_s c_1
		= g(s) \sigma(s) j^{1, k}_s c_1,
\]
where in the last step we have introduced the quantity $\sigma(s) :=
g(s)^{-1} \delta g(s) \in \mathfrak{o}(n)$. For the variation $\delta
\Omega$, we start with the definition $\Omega_\epsilon =
g^{-1}_\epsilon (s) g'_\epsilon(s)$, and take the derivative with
respect to $\epsilon$ to obtain
\[
	\delta \Omega = (\delta g^{-1}) g' + g^{-1} \delta g'
		= - g^{-1} \delta g g^{-1} g' + g^{-1} \delta g'.
\]
Noting that $\sigma' = - g^{-1} g' g^{-1} \delta g + g^{-1} \delta g'$,
this expression can be rewritten to yield
\[
	\delta \Omega  = \sigma' - \sigma \Omega + \Omega \sigma
		 = \sigma' - [\sigma, \Omega].
\]
This expression is familiar from classical mechanics, where it appears
in Euler-Poincar\'e reduction theory (see~\cite{MaRa1994}) or under
the guise of Lin constraints (see~\cite{CeMa1987}).

With these two expressions, we may now rewrite the expression for the
variation of $E$ as
\begin{equation} \label{eq:variation_E}
	\delta E = \int_I \Big(
		  \inner{ Q }{ g \sigma (j^{1,k} c_1) }_L
		+ \inner{ \Omega }{ \sigma' - [\sigma, \Omega] }_{\mathfrak{o}(n)}
		\Big) \d s.
\end{equation}
The first term may be expressed in terms of the diamond
operator~\eqref{eq:diamond_operator_definition} as
\[
  \inner{Q}{g \sigma (j^{1,k} c_1)}_L =
  \inner{g^{-1} Q}{\sigma (j^{1,k} c_1)}_L =
  \inner{j^{1,k} c_1 \diamond (g^{-1} Q)^\flat}{\sigma}_{\mathfrak{o}(n)}.
\]
Using the antisymmetry of the diamond map
(lemma~\ref{lem:antisymmetry}), we now see that
\[
	j^{1,k} c_1 \diamond (g^{-1} Q)^\flat
		= j^{1,k} c_1 \diamond (j^{1, k} c_1 - g^{-1} j^{1,k} c_2)^\flat
		= (g^{-1} j^{1,k} c_2) \diamond (j^{1,k} c_1)^\flat,
\]
so that
\[
	\inner{ Q }{ g \sigma (j^{1,k} c_1) }_L =
  \inner{ g^{-1} j^{1,k} c_2 \diamond (j^{1,k} c_1)^\flat }{ \sigma }_{\mathfrak{o}(n)}.
\]

To simplify the second term in~\eqref{eq:variation_E}, observe that
$\inner{ \Omega }{ [\sigma, \Omega] }_{\mathfrak{o}(n)} = 0$ for all
$\sigma, \Omega \in \mathfrak{o}(n)$.
This can be verified by a quick calculation, or
by noting that this is a consequence of the $\Ad$-invariance
of the inner product~\eqref{eq:killing_inner_product} on
$\mathfrak{o}(n)$.

Putting all of these results together, we then arrive at
\begin{align*}
	\delta E
  & = \int_I \Big(
			  \inner{ (g^{-1} j^{1,k} c_2)
					\diamond (j^{1,k} c_1)^\flat }{ \sigma }_{\mathfrak{o}(n)}
			+ \inner{ \Omega }{ \sigma' }_{\mathfrak{o}(n)} \Big) \d s \\
	& = \int_I
			 \inner{ (g^{-1} j^{1,k} c_2)
					\diamond (j^{1,k} c_1)^\flat -
				\Omega' }{ \sigma }_{\mathfrak{o}(n)} \d s +
			\inner{ \Omega(1) }{ \sigma(1) }_{\mathfrak{o}(n)} -
			\inner{ \Omega(0) }{ \sigma(0) }_{\mathfrak{o}(n)},
\end{align*}
where we have integrated by parts to obtain the second expression.

In order for a curve $g(s)$ to be a critical point of $E$, the
preceding expression must vanish for all variations $\sigma$, so that
we arrive at the following theorem.
\begin{thm}\label{thm:var_equations}
    A curve $g \in H^2(I,O(n))$ is a critical point of the
    energy functional~\eqref{eq:energy_functional} if and only if it
    satisfies the equation
    \begin{equation} \label{eq:omega_equations}
      \Omega'(s) = (g(s)^{-1} j^{1,k}_s c_2) \diamond (j^{1,k}_s c_1)^\flat
    \end{equation}
    with boundary conditions $\Omega(0) = \Omega(1) = 0$ and
    $\Omega(s) = g(s)^{-1} g'(s)$.
\end{thm}
Note that the equations~\eqref{eq:omega_equations} are second-order
differential equations when expressed in terms of $g(s)$. The boundary
conditions at each end can be written as $g'(0) = g'(1) = 0$, so that
we obtain a two-point boundary value problem for $g(s)$.

\section{Examples}
\label{sec:examples}

In this section we explicitly calculate minimizing curves and
distances for a few simple families of curves. First, let us consider
the most simple case of curves in the plane and only taking into
account first derivatives, i.e.~$n=2$ and~$k=1$. Let
$\theta \in \R$ parametrize the group\footnote{%
  We only consider $SO(n)$, the connected component of the identity in
  $O(n)$ here, thus disallowing orientation reversing minimizers.%
} $SO(2)$ in the usual sense,
\begin{equation}\label{eq:param-SO2}
  R(\theta) =
  \begin{pmatrix}
    \cos(\theta) & -\sin(\theta) \\
    \sin(\theta) &  \cos(\theta)
  \end{pmatrix} \in SO(2).
\end{equation}
Thus, $R\colon \R \to SO(2)$ is a Lie group homeomorphism, inducing
the Lie algebra isomorphism
\begin{equation}\label{eq:param-so2}
  \D_0 R\colon \R \to \mathfrak{so}(2), \qquad
  \omega \mapsto
  \begin{pmatrix}
    0      & -\omega \\
    \omega & 0
  \end{pmatrix}.
\end{equation}
For tangent vectors $A,B \in \R^2$,
expression~\eqref{eq:expression_diamond} for the diamond operator then
reduces to
\begin{equation*}
  A \diamond B^\flat = \frac{\lambda_1}{2}
  \begin{pmatrix}
    0                 & A_1 B_2 - B_1 A_2 \\
    B_1 A_2 - A_1 B_2 & 0
  \end{pmatrix} \in \mathfrak{so}(2),
\end{equation*}
recalling that $\lambda_1$ denotes the parameter of the inner product on
jets, that is, the strength of the `soft constraint'.
With the isomorphism $\D_0 R$ this allows us to write the variational
equation~\eqref{eq:omega_equations} for $\theta(s)$ as
\begin{equation*}
  \theta''(s) =
  (\D_0 R)^{-1}\big( R(\theta(s))^{-1} c_2'(s) \diamond c_1'(s)^\flat\big).
\end{equation*}
Since
$(\D_0 R)^{-1}(A \diamond B^\flat) = -\frac{\lambda_1}{2}(A \times B)$,
where $A \times B$ denotes the determinant of the matrix formed by the
column vectors $A,B$, we can now explicitly write the variational
equation as
\begin{equation}\label{eq:omega_eq_2d}
  \theta''(s) =
  -\frac{\lambda_1}{2}\big(R(-\theta(s)) c_2'(s) \times c_1'(s)\big)
\end{equation}
with boundary conditions $\theta'(0) = \theta'(1) = 0$.

\paragraph{Two straight lines.}
As a first example we consider two straight lines, but possibly
parametrized at nonconstant speed, that is, we consider the family of
curves of the form
\begin{equation}\label{eq:straight-line}
  c(s) = a\,f(s) + b \qquad \text{with } a,b \in \R^2
\end{equation}
and $f \in C^1(I,\R)$ with $f'(s) > 0$ for all $s \in I$.
Note that by using invariance under rigid transformations and by
absorbing the length of $a$ into the
parametrization $f$, we can bring these into the normalized form
$c(s) = f(s)\,e_1$, where $e_1$ is the first standard basis vector.
Taking two such curves $c_1, c_2$, we see that $\theta(s) = 0$
solves~\eqref{eq:omega_eq_2d} with boundary conditions, since
\begin{equation*}
  c_1'(s) \times c_2'(s) = f_1'(s)\,f_2'(s)\, e_1 \times e_1 = 0.
\end{equation*}
We find a corresponding energy
\begin{equation}\label{eq:straight-lines-energy}
  E = \int_0^1 \norm{(f_1'(s) - f_2'(s)) e_1}^2 \d s.
\end{equation}
The kinetic term $\norm{g(s)^{-1}g'(s)}$ is already zero and the
potential term above is as small as possible, since the vectors are
already aligned, hence $\theta(s) = 0$ is the minimizer.
Equation~\eqref{eq:straight-lines-energy} also shows that the distance is
minimal when both straight lines are parametrized at constant speed.
This can be seen by writing $f_i'(s) = a_i + g_i(s)$, where $a$ is the
average velocity, and noting that $a$ and $g$ are perpendicular as $L^2$
functions, thus reducing~\eqref{eq:straight-lines-energy} to
$E = \abs{a_1 - a_2}^2 + \norm{g_1 - g_2}_{L^2}^2$.

\paragraph{A line and a circle.}
Next, for a line $c_1$ and a circle,
$c_2(s) = r(\cos(2\pi s),\sin(2\pi s))$, we obtain
\begin{equation*}
  \theta''(s) = \pi\,\lambda_1\,r\,\cos(2\pi s - \theta(s)).
\end{equation*}
After a coordinate substitution $\theta(s) = 2\pi s - \phi(s)$ we
obtain the pendulum equation
\begin{equation}\label{eq:pendulum}
  \phi''(s) = -\pi\,\lambda_1\,r\,\cos(\phi(s))
\end{equation}
with boundary conditions $\phi'(0) = \phi'(1) = 2\pi$. For
non-overturning oscillations the period is bounded below by
$T_0 = 2\sqrt{\pi/(\lambda_1 r)}$, thus we only expect to see such
solutions when $\lambda_1 \gg 1$, i.e.\ in the regime of a strong
constraint. After resubstituting $\theta$ again, such solutions
correspond to rotating the constant tangent vector of $c_1$
approximately into the rotating tangent vector of the circle $c_2$,
that is, the winding number of $\theta$ is one. For small $\lambda_1$
the kinetic term dominates and gives a minimizer $\theta(s)$ with zero
winding number, see also Figure~\ref{fig:minimizers-line-circle}.
Numerical simulations indicate that the bifurcation takes place at
$\lambda_1 \approx 48.9$ and $E \approx 152$.
\begin{figure}[htb]
  \centering
  \includegraphics[width=7.5cm]{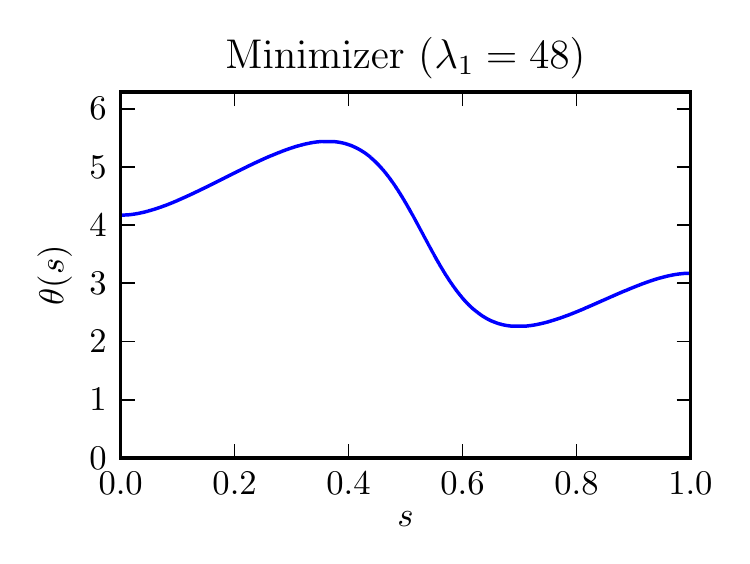}~
  \includegraphics[width=7.5cm]{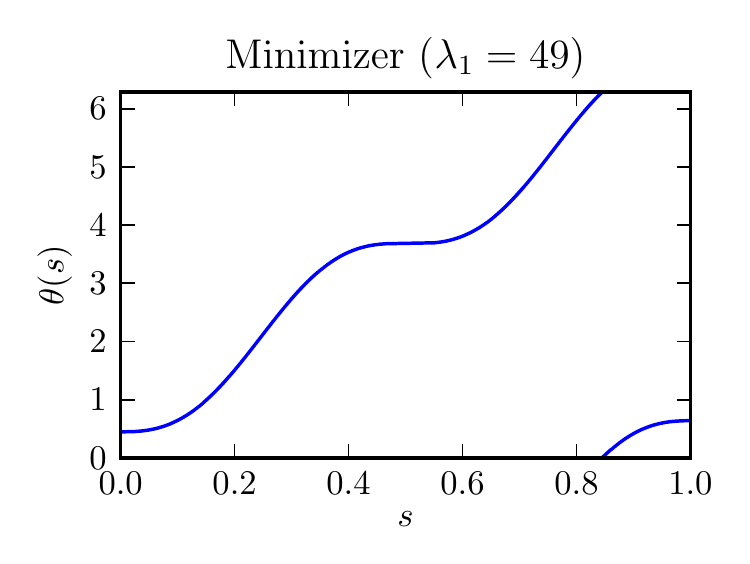}
  \caption{The minimizer $\theta(s)$ for the distance between a line
    and a circle at $\lambda_1 = 48$ and $\lambda_1 = 49$.}
  \label{fig:minimizers-line-circle}
\end{figure}

\paragraph{More general curves.}
If we consider curves beyond these simple examples, then we quickly
run into boundary value problems that do not have explicit solutions
anymore. For example, comparing a straight line $c_1$ to a curve that
is a graph, $c_2(s) = (s,h(s))$, we find that~\eqref{eq:omega_eq_2d}
reduces to
\begin{equation*}
  \theta''(s) = -\frac{\lambda_1}{2}\big(\sin(\theta(s)) - \cos(\theta(s)) h'(s)\big).
\end{equation*}
This is reminiscent of the pendulum equation~\eqref{eq:pendulum},
except that the magnitude and direction of gravity now explicitly
depend on $s$ through $h'(s)$. Although analytical solutions are out
of reach here, these equations can easily be solved numerically (also
in higher dimensions) and minimizers can be found by using adjoint
equations to solve the boundary value problem with a Newton--Raphson
method.

\paragraph{Equations in $\R^3$.}
Finally, let us express the equations~\eqref{eq:omega_equations} in
three dimensions, first with $k = 1$. The identification of Lie
algebra elements $\Omega \in \mathfrak{o}(n)$ and
$\omega \in (\R^3, \times)$ via
$\Omega_{ij} = \epsilon_{ijk}\,\omega_k$ allows us to identify
$\Omega = A \diamond B^\flat$ with
\begin{equation*}
  \omega_k
  = \frac{1}{2} \epsilon_{ijk} \Omega_{ij}
  = \frac{\lambda_1}{2} \epsilon_{ijk} A_i B_j.
\end{equation*}
Equation~\eqref{eq:omega_equations} then becomes
\begin{equation}\label{eq:omega_eq_3d}
  \omega'(s) = \frac{\lambda_1}{2} \big(g(s)^{-1}\cdot c_2'(s) \times c_1'(s)\big)
\end{equation}
where $g(s) \in O(3)$ and $\times$ now denotes the usual cross product
on $\R^3$. Note that this generalizes to arbitrary $k$ as
\begin{equation*}
  \omega'(s) = \frac{1}{2} \sum_{i=1}^k \lambda_i
                 \big(g(s)^{-1}\cdot c_2^{(i)}(s) \times c_1^{(i)}(s)\big).
\end{equation*}

\section{The Existence of a Minimizer}
\label{sec:exist-minimizer}

We now prove that there exists a minimizer for the curve-matching
functional. Our approach follows closely the proof of minimizers for
the LDDMM functional in~\cite{BrHo2013,Yo2010}, while also using some
theory on Hilbert manifolds, see~\cite{Pa1963,Kl1995}.

Let $O(n)$ be the orthogonal group in $n$ dimensions, with Lie algebra
$\mathfrak{o}(n)$. In~\eqref{eq:killing_inner_product} we defined the
inner product $\inner{\xi}{\eta}_{\mathfrak{o}(n)} := - \tr(\xi\eta)$.
As this is a positive definite $\Ad$-invariant quadratic form, it
determines a bi-invariant Riemannian metric on $O(n)$ with induced
distance $d(\slot,\slot)_{O(n)}$. Therefore,
$L^2(I,\mathfrak{o}(n))$ is a Hilbert space and $H^1(I,O(n))$ can be
viewed as a $C^\infty$ Hilbert manifold~\cite[Thm.~2.3.12]{Kl1995}
(see also~\cite[\S 13, Thm.~6]{Pa1963} but note that Palais'
definition of manifold structure is slightly different due to the use
of an embedding).

We briefly recall Definition~\cite[Def.~2.3.1]{Kl1995} of the Hilbert
manifold $H^1(I,M)$ where $M$ is a complete Riemannian manifold: it
consists of the curves $c \in C^0(I,M)$ such that for any chart
$(\phi,U)$ of $M$ and $I' = c^{-1}(U)$ we have
$\phi \circ c \in H^1(I',\R^n)$. Canonical charts for $H^1(I,M)$ are
given by the exponential map along piecewise smooth curves $c$.

This Hilbert manifold comes equipped with a natural
inner product which turns it into an infinite-dimensional Riemannian
manifold. Let $c \in H^1(I,M)$ and let
$\lambda,\mu \in H^1(c^*(\T M))$ denote sections in the pullback
bundle, i.e.~$\lambda(s),\mu(s) \in \T_{c(s)} M$ for all $s \in I$.
This pullback bundle serves as a natural chart for the tangent space
$\T_c H^1(I,M)$. The Riemannian metric on $H^1(I,M)$ is given in this
chart by
\begin{equation}\label{eq:metric_H1}
  \inner{\lambda}{\mu}_{H^1(I,M)}
  = \int_I \Big( \inner{\lambda(s)}{\mu(s)}_M
               + \inner{\nabla\lambda(s)}{\nabla\mu(s)}_M \Big) \d s,
\end{equation}
where $\nabla$ denotes the covariant derivative pulled back from
$\T M$ to $c^*(\T M)$.

Having this preliminary theory at hand, we can now state the following
result in the case that $M = O(n)$.
\begin{thm}\label{thm:exist-minimizer}
  The energy functional $E\colon H^1(I,O(n)) \to \R$ defined
  by~\eqref{eq:energy_functional} attains its global minimum on
  $H^1(I,O(n))$. Moreover, any minimizer $g^* \in H^1(I,O(n))$
  satisfies the variational equations~\eqref{eq:omega_equations} and
  in particular $g^* \in C^2(I,O(n))$.
\end{thm}

We first identify the space $H^1(I,O(n))$ with a simpler one. This
identification we can make in general, for any Lie group $G$ equipped
with a left (or right) invariant inner product, such that $G$ is a
Riemannian manifold.

\begin{prop}\label{prop:equiv-hilbert-mflds}
  Let $G$ be a Lie group equipped with a left-invariant inner product
  and let $\mathfrak{g}$ denote its Lie algebra. There is a natural
  bijection $\phi\colon H^1(I,G) \to G \times L^2(I,\mathfrak{g})$
  given by
  \begin{equation}
    \phi(g) = (g(0), s \mapsto g(s)^{-1} g'(s)),
  \end{equation}
  and $g = \phi^{-1}(g_0,\xi)$ is given by the reconstruction equation
  \begin{equation}\label{eq:reconstruction}
    g'(s) = g(s) \xi(s) \quad \text{with initial conditions } g(0) = g_0.
  \end{equation}
\end{prop}

\begin{proof}
  For $g \in H^1(I,G)$ it follows that
  $s \mapsto \xi(s) = g(s)^{-1} g'(s) \in L^2(I,\mathfrak{g})$
  since $g' \in L^2(I,\T O(n))$ by definition, and left-translation by
  $g(s)^{-1}$ is a linear isomorphism and continuous (thus bounded) in
  $s$, hence it is a bounded operator on $L^2(I,\T G)$ that maps
  $g(s)^{-1} g'(s) \in \T_e G = \mathfrak{g}$. Clearly, also
  $g(0) \in G$. Hence, $\phi$ is well-defined.

  To prove that $\phi$ is a bijection, we show that $\phi^{-1}$ is
  well-defined and indeed the inverse of $\phi$. Viewing $\xi(s)$ as a
  left-invariant vector field on $G$ turns~\eqref{eq:reconstruction}
  into a Carath\'eodory type\footnote{%
    Ordinary differential equations that are only integrable in time
    are called Carath\'eodory type differential equations. Basically,
    they still exhibit existence and uniqueness of solutions. For more
    details see~\cite[Chap.~2]{Coddington1955:theory-ODEs}
    or~\cite[App.~C]{Yo2010}.%
  } differential equation on $G$, since
  $\xi \in L^2(I,\mathcal{X}(G))$ with respect to local charts
  $(\psi,U)$ on $G$. By left-invariance of the metric on $G$ it
  follows that the left-invariant vector field $\xi$ is bounded by an
  integrable function:
  \[
    \int_I \sup_{g \in G} \norm{\xi(s,g)} \d s
  = \int_I \norm{\xi(s)} \d s
  =   \norm{\xi}_{L^1(I,\mathfrak{g})}
  \le \norm{\xi}_{L^2(I,\mathfrak{g})} < \infty.
  \]
  Theorem~2.1 in~\cite[Chap.~2]{Coddington1955:theory-ODEs} then
  implies local existence and uniqueness of the solution
  of~\eqref{eq:reconstruction} with respect to a chart $(\psi,U)$. By
  Theorem~1.3 in~\cite[Chap.~2]{Coddington1955:theory-ODEs} these
  solutions extend to the boundary of $\psi(U)$, hence they can be
  patched together to a maximal solution $g\colon I \to G$. Solution
  curves are by construction $H^1$ functions with respect to these
  charts and in view of the definition of $H^1(I,G)$, this implies
  that $\phi^{-1}$ is well-defined. Existence and uniqueness now
  implies that $\phi^{-1}$ is the inverse of $\phi$.
\end{proof}

\begin{lemma} \label{lem:phi-inv-cont}
  Let $C^0(I,O(n))$ be equipped with the supremum distance
  \begin{equation}\label{eq:sup-dist}
    d(g,h)_{\sup} = \sup_{s \in I}\; d(g(s),h(s))_{O(n)}.
  \end{equation}
  The map $\phi^{-1}$ is continuous into $C^0(I,O(n))$; moreover it
  maps bounded, weakly convergent sequences onto convergent sequences.
\end{lemma}

We summarize Theorems~8.7 and~8.11 from~\cite{Yo2010} below, and use
this result as the basic ingredient for the proof of the lemma.
\begin{thm}\label{thm:weak-conv-flow}
  Let $v^* \in L^1(I,C_0^1(\Omega,\R^d))$ be a time-dependent vector
  field with support in the open, bounded set $\Omega \subset \R^d$,
  and let $v_i \in L^1(I,C_0^1(\Omega,\R^d))$ be a bounded sequence,
  weakly convergent to $v^*$. Then the respective flows $\phi^*,\phi_i$
  are diffeomorphisms at all times, and $\phi_i \to \phi^*$ in
  supremum norm on $\bar{\Omega}$.
\end{thm}

\begin{proof}[Proof of Lemma~\ref{lem:phi-inv-cont}]
  Note that it is sufficient to prove the last, stronger claim. Let
  $\xi \in L^2(I,\mathfrak{o}(n))$ and $g_0 \in O(n)$, and let
  $g = \phi^{-1}(g_0,\xi)$. In order to apply
  Theorem~\ref{thm:weak-conv-flow}, we shall embed everything in
  linear spaces. Note that $O(n)$ has the usual isometric embedding
  into the space of matrices, $M_n \cong \R^{n \times n}$, and the
  Frobenius inner product corresponds to the normal Euclidean one. We
  interpret elements $\eta \in \mathfrak{o}(n)$ as left-invariant
  vector fields on $O(n)$. These can be smoothly extended to have
  compact support on an $O(n)$-invariant tubular neighborhood\footnote{%
    See~\cite[Lem.~2.30]{El2013} for the construction of an $O(n)$-invariant
    tubular neighborhood. This allows one to lift $\eta$ to
    a horizontal vector field on this neighborhood, multiplied by a
    smooth, $O(n)$-invariant cut-off function with compact support.
    This extended vector field is integrable when $\eta$ is.%
  } of $O(n) \subset M_n$. Thus,
  $\xi \in L^2(I,\mathfrak{o}(n))$ defines a time-dependent vector
  field on $M_n$ with compact support, that is smooth with respect to
  $g \in M_n$ and square-integrable with respect to $s \in I$.
  Since $I$ is bounded, this implies absolute integrability,
  satisfying the conditions of Theorem~\ref{thm:weak-conv-flow}.
  The extended vector field leaves $O(n)$ invariant and is identical
  to the original $\eta$ on $O(n)$, so the generated flow can be
  restricted to a flow on $O(n)$, which is the one generated by
  $\eta$.

  Next, we claim that there exists a $C \ge 1$ such that
  \begin{equation}\label{eq:equiv-dist-norm}
    C^{-1} d(g,g')_{O(n)} \le \norm{g - g'}_{M_n} \le C d(g,g')_{O(n)}.
  \end{equation}
  Note that the second inequality follows straightforwardly (with
  $C = 1$) from the isometric embedding $O(n) \subset M_n$. We sketch
  the argument for the first inequality, more details can be found
  in~\cite[Lem.~2.25]{El2013}. Using symmetry, we need only prove the
  case where $g' = e$. For a sufficiently small $\delta > 0$ and
  $d(e,g)_{O(n)} < \delta$, we can consider $B(e;\delta) \subset O(n)$
  and view it as the graph of the function
  $\exp\colon \mathfrak{o}(n) \to M_n$ which has identity derivative
  at $e$, thus we can bound $\norm{\D\exp - I} \le \frac{1}{2}$, say,
  on $B(e;\delta)$. Now we write
  \begin{equation*}
    g - e = \int_0^1 \D\exp(t\log(g))\cdot\log(g) \d t,
  \end{equation*}
  from which the estimate
  \begin{equation*}
    \norm{g - e}_{M_n} \ge \int_0^1 \frac{1}{2}\norm{\log(g)} \d t
    = \frac{1}{2} d(e,g)_{O(n)},
  \end{equation*}
  follows. For $d(e,g)_{O(n)} > \delta$ we can simply choose
  $C \ge \text{diam}_{M_n}(O(n)) / \delta$.

  Finally, we note that $L^2(I,M_n)$ is continuously embedded into
  $L^1(I,M_n)$ and apply Theorem~\ref{thm:weak-conv-flow} to conclude
  that $\phi^{-1}$ maps bounded, weakly convergent sequences into
  convergent sequences in $C^0(I,O(n)_{M_n})$, where $O(n)_{M_n}$
  denotes $O(n)$ with the distance induced by the embedding into
  $M_n$. By~\eqref{eq:equiv-dist-norm} this distance is equivalent to
  the intrinsic distance on $O(n)$, hence the result follows.
\end{proof}

\paragraph{The registration functional.}

We will prove the existence of a minimizer for a slightly more general
functional, given by
\begin{equation}\label{eq:energy_generalized}
  E(g) = V(g) + T(g) = \int_I \Big( U(s, g(s)) + \frac{1}{2}\norm{\xi(s)}^2 \Big) \d s,
\end{equation}
where $\xi(s) = g(s)^{-1} g'(s) \in \mathfrak{o}(n)$ and $V$ and $T$
denote the potential and kinetic parts of the functional,
respectively. The function $U\colon I \times O(n) \to \R$ is assumed
to be continuous, and the derivative with respect to the second
variable, $\D_2 U$, also continuous.


We take a minimizing sequence $g_i \in H^1(I,O(n))$, that is,
\begin{equation*}
  \lim_{i \to \infty} E(g_i) = \inf_{g \in H^1(I,O(n))} E(g).
\end{equation*}
Since $\lim_{\norm{g}_{H^1} \to \infty} E(g) = \infty$, it follows
that this sequence is bounded.
Using $\phi$ from Proposition~\ref{prop:equiv-hilbert-mflds} we can
identify this sequence with a sequence
$(g_0,\xi)_i \in O(n) \times L^2(I,\mathfrak{o}(n))$.

In the following we take subsequences without denoting these with new
indices. First, since $O(n)$ is compact, we can take a subsequence
such that $(g_0)_i$ converges to an element $g^*_0 \in O(n)$.
Secondly, since the sequence $\xi_i$ is bounded, we extract a
subsequence such that $\xi_i$ converges weakly
to some $\xi^* \in L^2(I,\mathfrak{o}(n))$. We show that
$g^* = \phi^{-1}(g^*_0,\xi^*)$ is a minimizer of $E$.

We have automatically that
\begin{equation}
  \inf_{g \in H^1(I,O(n))} E(g) \le E(g^*),
\end{equation}
and we now prove the reverse inequality. The only place in which our
proof differs from the sketched proof of~\cite[Thm.~21]{BrHo2013}
is in our treatment of the
terms involving $U$. Let $g_i$ and $g^*$ be the solutions of the
reconstruction equation~\eqref{eq:reconstruction} associated with
$(g_0,\xi)_i$ and $(g^*_0,\xi^*)$, respectively. By
Lemma~\ref{lem:phi-inv-cont} we have that $g_i \to g^*$ under
the supremum norm. The function $U$ is bounded since it is continuous
on a compact domain; by the bounded convergence theorem we have
therefore that
\[
	\lim_{i \to \infty} \int_I U(s, g_i(s)) \d s = \int_I U(s, g^*(s)) \d s,
\]
that is, the functional defined by $U$ is continuous.
The rest of the proof now proceeds as in~\cite{BrHo2013}: from the inequality
\[
  \inner{\xi_i}{\xi^*}_{L^2} \le
  \norm{\xi_i}_{L^2} \norm{\xi^*}_{L^2},
\]
we have that $\norm{\xi^*}_{L^2} \le \liminf_{i \to \infty} \norm{\xi_i}_{L^2}$,
and therefore
\begin{align*}
  E(g^*)
  &=   \frac{1}{2} \norm{\xi^*}_{L^2}^2 +
       \int_I U(s, g^*(s)) \d s \\
  &\le \liminf_{i \to \infty} \frac{1}{2} \norm{\xi_i}_{L^2}^2 +
          \lim_{i \to \infty} \int_I U(s, g_i(s)) \d s \\
	&=      \lim_{i \to \infty} E(g_i)
	 = \inf_{g \in H^1(I,O(n))} E(g).
\end{align*}
This concludes the proof that the energy functional has a minimizer.

Finally, we can conclude that any minimizer of the energy functional
is a solution of the variational equations~\eqref{eq:omega_equations}
as follows. First, by Proposition~\ref{prop:E-differentiable} below
$E$ is differentiable, hence at a minimizer $g \in H^1(I,O(n))$ we
must have $\D E(g) = 0$; this is equivalent
to~\eqref{eq:variation_E}. The variational equations derived from this
can be interpreted in a distributional sense, acting on variations
$\var g \in C^\infty(I,O(n))$. However, the
equations~\eqref{eq:omega_equations} define a continuous vector field
on $\T O(n)$ coming from a second order ODE, so solutions of it must
be curves $g \in C^2(I,O(n))$, see for instance~\cite[p.~178]{Du1976}
or~\cite[Thm.~4.1]{Buttazzo1998:1d-var-problems}.
This completes the proof of Theorem~\ref{thm:exist-minimizer}.

\begin{prop}\label{prop:E-differentiable}
  The functional $E\colon H^1(I,O(n)) \to \R$ given
  by~\eqref{eq:energy_generalized} is differentiable.
\end{prop}

\begin{proof}
  Differentiability of the kinetic term $T(g)$ follows
  from~\cite[Thm.~2.3.20]{Kl1995}. For the potential term $V(g)$, we
  claim that the derivative is given by
  \begin{equation}\label{eq:V-derivative}
    \D V(g) \var g = \int_I \D_2 U(s,g(s)) \var g(s) \d s.
  \end{equation}
  Since $\D_2 U$ is continuous on a compact domain, it is uniformly
  continuous; let $\epsilon$ denote its modulus of continuity. We
  directly estimate
  \begin{align*}
    &\hspace{-1cm}\abs[\big]{V(g+\var g) - V(g) - \D V(g) \var g}\\
    &=   \int_I \abs[\big]{U(s,g(s)+\var g(s)) - U(s,g(s)) - \D_2 U(s,g(s)) \var g(s)} \d s\\
    &\le \int_I \int_0^1 \abs[\big]{\D_2 U(s,g(s)+\tau \var g(s)) \var g(s)
                                  - \D_2 U(s,g(s))                \var g(s)} \d\tau \d s\\
    &\le \epsilon\big(\sup_{s \in I} \norm{\var g(s)}\big) \int_I \norm{\var g(s)} \d s
     \in o(\norm{\var g}_{H^1}),
  \end{align*}
  which proves our claim.
\end{proof}

\section{Probabilistic approach to curve matching}
\label{sec:probability}

The energy functional~\eqref{eq:energy_functional} can be derived from
a Bayesian point of view as well. To see this, we adapt the Bayesian
approach to linear regression (see for instance~\cite{Bi2006}) to the
case of curve matching. We first make the following simplifications:
\begin{enumerate}
\item We assume that the curves are planar and we let the order of the
  jets be $k = 1$, so that only the tangent vector field of the curves
  is taken into account.
\item We only consider the tangent vector field at the end points of
  $N$ regularly spaced intervals in the parameter $s$. In other words,
  we sample $c_1'(s)$ and $c_2'(s)$ at $s_n = n \Delta s$ for $n = 0,
  \ldots, N$, where $\Delta s = 1 / N$ and $n = 0, \ldots N$.  Later
  on, we will let $N$ approach infinity.
\end{enumerate}

Throughout the remainder of this paragraph, we will use
$\{c_1'(s_n)\}$ as a shorthand for the ensemble $\{c_1'(s_0), \ldots,
c_1'(s_N)\}$, and similarly for $\{c_2'(s_n)\}$.

Now let $c_1(s_n)$ be fixed, and assume that, for each $n = 0, \ldots, N$,
$c_2'(s_n)$ is found by acting on $c_1'(s_n)$ with a rotation matrix $g_n$,
and by adding noise:
\[
   c_2'(s_n) = g_n \cdot c_1'(s_n) + \epsilon_n,
\]
where the $\epsilon_n$ are $\mathbb{R}^2$-valued, independently
distributed Gaussian random variables with mean $0$ and variance
$\sigma^2 / \lambda I$, where $I$ is the $2 \times 2$ identity
matrix. The constant $\lambda$ will play the same role as the scaling
parameters in the norm~\eqref{eq:jet_inner_product}.

The conditional probability to obtain the ensemble $\{c_2'(s_n)\}$ for
$n = 0, \ldots, N$, given $\{c_1'(s_n)\}$ and $\{g_n\}$ is
\[
  p(\{c_2'(s_n)\} \, | \, \{c_1'(s_n)\}, \{g_n\}) \propto
    \exp\left( - \frac{\lambda}{\sigma^2} \Delta s
      \sum_{n = 0}^N \norm{c_2'(s_n)  - g_n \cdot c_1'(s_n)}^2
\right).
\]
Now assume that we choose a prior on the space of discrete curves
$\{g_n\}$ which privileges curves which are nearly constant. For instance,
we may choose
\[
  p(\{g_n\}) \propto \exp \left( - \frac{1}{\Delta s}
    \sum_{n = 0}^{N-1} \norm{g_{n + 1} - g_n}^2 \right),
\]
where $\norm{\,\cdot\,}$ is the Frobenius norm on the
space of matrices. By Bayes' theorem, we can then calculate the
conditional probability for $\{g_n\}$ given $\{c_1'(s_n)\}$ and $\{c_2'(s_n)\}$ as
\[
  p(\{g_n\} \, | \, \{c_1'(s_n)\}, \{c_2'(s_n)\}) \propto
      p(\{c_2'(s_n)\} \, | \, \{c_1'(s_n)\}, \{g_n\}) p(\{g_n\}),
\]
where we have used the fact that $g_n$ is independent from
$c_1'(s_n)$, so that $p(\{g_n\} | \{c_1'(s_n)\}) = p(\{g_n\})$. The negative
logarithm of this density is given by
\begin{multline} \label{logpdf}
- \log p(\{g_n\} \, | \, \{c_1'(s_n)\}, \{c_2'(s_n)\}) = \\
 \frac{\lambda}{\sigma^2} \Delta s
      \sum_{n = 0}^N \norm{c_2'(s_n)  - g_n \cdot c_1'(s_n)}^2
+ \frac{1}{\Delta s}
    \sum_{n = 0}^{N-1} \norm{g_{n + 1} - g_n}^2,
\end{multline}
and for $N \to \infty$, this gives precisely the matching
energy~\eqref{eq:energy_functional}. From this point of the view, the
curve $g(s)$ that minimizes~\eqref{eq:energy_functional} is precisely
the maximum posterior estimate of (the continuum version of) the
distribution~\eqref{logpdf}.

\section{Conclusions and outlook}

In this paper, we have defined a notion of distance between
parametrized curves in $\R^n$, and we have shown (among other
things) that --- in contrast to previous approaches --- this notion is
a true distance function. Our approach uses only standard geometrical
notions, and is hence eminently generalizable. Below, we describe some
directions for future research.

\paragraph{Statistical analysis of curves and shapes.} Using the
distance function on the space of curves defined in this paper (and
its generalization to surfaces, described below), we can embark on a
statistical analysis of shapes and curves. Having a notion of distance
will allow us to register curves, compute (Fr\'echet) means, and use
tools from Riemannian geometry in the exploration of shape geometries;
see~\cite{LePennec2006} for more details.

\paragraph{Matching of surfaces.} The method presented in this paper
naturally generalizes to matching of parametrized (2D) surfaces in
$\R^n$. We sketch here the setup; for simplicity of presentation we
set $k = 1$ and take as parametrization domain the torus
$\mathbb{T}^2 = [0,1]^2 / \sim$, where $\sim$ identifies opposite
edges of the unit square. Then the energy functional
becomes\footnote{%
  We have to add the condition that $g \in L^\infty$, since this is
  not implied anymore by Sobolev inequalities if $g \in H^1$ on a
  two-dimensional domain.%
}
\begin{equation}\label{eq:energy_surfaces}
  E[g;f_1,f_2] = \frac{1}{2} \int_{[0,1]^2}
                   \norm{g(s,t)\cdot\D f_1(s,t) - \D f_2(s,t)}^2
                  +\norm{g(s,t)^{-1} \D g(s,t)}^2  \d s \d t,
\end{equation}
with $g\colon \mathbb{T}^2 \to O(n)$ and
$f_i\colon \mathbb{T}^2 \to \R^n$, $\D$ denotes a derivative with
respect to both variables $s,t$ and the norms are now defined on pairs
of elements in $V$ and $\mathfrak{o}(n)$ respectively, by a
square-root of the sum of squares. Finally, $g(s,t)$ simply acts on
each element of the pairs. The variation of the functional, $\var E$,
now contains a derivative both with respect to $s$ and $t$.
Integrating each by parts, this leads to the variational equations
\begin{equation}\label{eq:var_eq_surfaces}
  \D_s \Omega_s + \D_t \Omega_t
  = (g^{-1} \cdot \D f_2) \diamond (\D f_1)^\flat,
\end{equation}
where everything depends on $(s,t)$ and the subscript $s,t$ denote
derivatives and components with respect to these. This equation
naturally generalizes~\eqref{eq:omega_equations}, although it is now a
second order partial differential equation. To decompose this into
first order PDEs, one has to add the equation
\begin{equation}\label{eq:curvature_relations}
  \D_s \Omega_t - \D_t \Omega_s = \big[ \Omega_t \,,\, \Omega_s \big]
\end{equation}
that expresses symmetry of the second order derivatives of the underlying
$g(s,t)$. This formulation is closely related to that of a $G$-strand
and~\eqref{eq:curvature_relations} is known as the
\emph{zero curvature relation}, see~\cite{HoIvPe2012}.

\paragraph{Generalizations to curved manifolds.} From a geometric
point of view, it would be interesting to consider the extension of
this framework to the case of curves taking values in an arbitrary
Riemannian manifold $M$. One immediate difficulty is that the matching
term in~\eqref{eq:energy_functional} involves the difference of jets
at different points. To remedy this, one could either choose a flat
background connection, if possible, and parallel translate the jets to
a common base point. Another possibility would be to replace the group
$O(n)$ by the \emph{orthogonal groupoid} $O(\T M, \T M)$, consisting
of linear bundle isometries from $\T M$ to $\T M$. If we view the
combined source and target maps
$(\alpha,\beta)\colon O(\T M, \T M) \to M \times M$ as a projection
defining a fiber bundle, then $g \in H^1(I,O(n))$ is generalized to
being a section
\begin{equation*}
  g \in H^1\big((c_1,c_2)^* O(\T M, \T M)\big),
\end{equation*}
that is, we have $g(s) \in O(\T_{c_1(s)} M,\T_{c_2(s)} M)$. The action
of $g(s)$ on a jet $j^k_s c_1$ can be defined by its action on the
covariant derivatives of $c_1$. Parallel transported orthonormal
frames along the curves $c_i$ will allow representing $g'(s)$ in
$\mathfrak{o}(n)$.

\section*{Acknowledgements}

We would like to thank Martin Bauer, Martins Bruveris, Darryl Holm and
Lyle Noakes for valuable comments and discussions. Both authors are
supported by the ERC Advanced Grant 267382.
JV is also grateful for partial support by the Irses project GEOMECH
(nr.~246981) within the 7th European Community Framework Programme, and
is on leave from a Postdoctoral Fellowship of the Research
Foundation-Flanders (FWO-Vlaanderen).

\printbibliography

\end{document}